\documentclass[12pt]{article}
\usepackage{amsmath,amsfonts,amssymb,amsthm,amscd}
\usepackage{xcolor}
\title{Remarks on relative categoricity}
\date{\today}
\author{Anand Pillay\thanks{Supported by NSF grants  DMS-2054271 and DMS-2502292}\\{University of Notre Dame}}

\newtheorem{Theorem}{Theorem}[section]
\newtheorem{Proposition}[Theorem]{Proposition}
\newtheorem{Definition}[Theorem]{Definition}
\newtheorem{Remark}[Theorem]{Remark}
\newtheorem{Lemma}[Theorem]{Lemma}
\newtheorem{Corollary}[Theorem]{Corollary}
\newtheorem{Fact}[Theorem]{Fact}

\newtheorem{Conjecture}[Theorem]{Conjecture}

\newcommand{\Z}{\mathbb Z}

\begin{document}
\maketitle

\begin{abstract}  This paper is partly a survey with historical background and references, partly provides the opportunity to put in print some unpublished early work of mine, and partly has some new results. 
Precise definitions will be given below, but roughly speaking $T$ will be a complete theory in a countable language $L$  with a distinguished unary predicate $P$.  $T$ is relatively categorical if any isomorphism between the $P$-parts of two models of $T$ lifts to an isomorphism between the two models.  It was conjectured that if $T$ is relatively categorical then any model of $T^{P}$ (the relativization of $T$ to $P$) is the $P$-part of a model of $T$ (the Gaifman property, also called``existence").  This remains open.  

We identify a special case of relative categoricity, namely  when $T$ is an ``almost internal cover" of $T^{P}$ and prove the Gaifman conjecture in this case. We introduce the notion of (full) $\omega$-stability over $P$ and show it implies the existence of constructible models over all ``complete" sets. We give an example of $T$ which is relatively categorical and an  internal cover of $T^{P}$ but is not  fully $\omega$-stable over $P$, fact not even fully stable over $P$ in the sense of \cite{Usvyatsov}. 
\end{abstract}

\section{Introduction and preliminaries}
The subject of relative categoricity was started by Gaifman in \cite{Gaifman}  with his ``single-valued operations", which was aimed at trying to find a logical framework for understanding and studying 
certain mathematical constructions.  One example is passing from an integral domain $R$ to its field of fractions $Frac(R)$. Here $Frac(R)$ is explicitly definable from $R$, in fact interpretable in $R$.  Another is passing from a field $F$ to an $n$-dimensional 
vector space $V$ over $F$.  Here  $V$ is what we call ``internal" to $F$: it is interpretable in $F$ only after fixing or naming a basis for $V$.  The case of explicit definability or interpretability coincides with Hodges' word constructions \cite{Hodges-word}. 

The actual definition of these ``Gaifman operations" involved rather ``implicit" definability.  The original definition involved two countable languages $L^{-}\subseteq L$ and a unary predicate symhol $P\in L\setminus L^{-}$. But it is easy to simplify the presentation as follows:  From now on we take  $T$ to be a complete theory  which has quantifier elimination in a countable relational language $L$. And $P$ will be a distinguished unary predicate symbol in $L$.  For $M\models T$, $M^{P}$ denotes the  $L$-substructure of $M$ with universe $P(M)$, and $T^{P}$ denotes the common complete theory of these $M^{P}$ as $M$ ranges over models of $T$.  
We assume that $T$ says that $P$ is infinite (to avoid trivialties). Note that the so called uniform reduction property is now built in to the set-up:  For any $L$-formula $\phi(x)$, there is an $L$-formula  $\psi(x)$ such that for any model 
$M$ of $T$ and tuple $a$ from $P(M)$, $M\models \phi(a)$ iff $M^{P}\models \psi(a)$. (Let $\psi$ be the quantifier-free formula which is equivalent to $\phi$ in $T$.)
$\kappa$ denotes an infinite cardinal. 
\begin{Definition} (i) $T$ is relatively categorical  if whenever $M_{1}, M_{2}$ are models of $T$ and $f$ is an isomorphism between $M_{1}^{P}$ and $M_{2}^{P}$ then $f$ lifts to an isomorphism between $M_{1}$ and $M_{2}$.
\newline
(ii) $T$ is relatively $\kappa$-categorical if the right hand side of (i) holds under the assumption that $M_{1}^{P}$ (and $M_{2}^{P}$) have cardinality $\kappa$.
\newline
(iii) $T$ is  $(\kappa, \kappa)$-categorical if the right hand side of (i) holds under the assumption that both $M_{i}$ and $M_{i}^{P}$ have cardinality $\kappa$ for $i=1,2$.
\end{Definition}

We will only really discuss $(\kappa,\kappa)$-categoricity when $\kappa = \omega$. 

\begin{Definition} $T$ has the Gaifman property (also called ``$P$-existence")  if for any $N\models T^{P}$ there is $M\models T$ such that $N = M^{P}$.
\end{Definition}

\begin{Conjecture} If $T$ is relatively categorical, then $T$ has the Gaifman property.
\end{Conjecture} 

This conjecture is   attributed to Haim Gaifman, the reference being the {\bf Problem} on p. 30 of \cite{Gaifman}, where the situation is compared to Beth's theorem (where no new sorts are added).   Actually the paper \cite{Gaifman} contains at the end a note added in  June 1978, with ``syntactic" characterizations of relative categoricity, under some additional assumptions, but with no proofs.  One of these additional properties is ``rigidity"; for every model $M$ of $T$, $Aut(M/P(M))$ is trivial.  The  syntactic characterization is that $M$ is explicitly definable from $M^{P}$ (as in the case of passing from an integral domain to its field of fractions),  namely for any model $M$ of $T$, $M = dcl_{M}(P(M))$ which therefore must happen uniformly.   A proof appears in \cite{Hodges-book}  (see Theorem 12.5.8 there),  but it is close to a triviality. In any case a positive answer to Conjecture 1.3 holds in this case.  Another property is ``weak rigidity", namely   for any $M\models T$,  $|Aut(M/P(M))| <  2^{|M|}$.  The syntactic characterization is that  any model $M$ of $T$ is explicitly definable from $P(M)$, but over additional parameters from $M$ (and happens uniformly). This is precisely what we now call ``internality".  I  give a proof later in this section. 
Later in this paper I will consider the slightly weaker property of ``almost internality" where every model of $T$ is in the algebraic closure of $P(M)$ together with additional parameters from $M$ (so uniformly so), where we obtain a positive solution to Conjecture 1.3.

In the first part of my Ph.D. thesis \cite{Pillay-thesis} I  proved Conjecture 1.3, using what Gaifman said in \cite{Gaifman-conjecture}  was a  theorem of Shelah (for $M_{1}, M_{2}$ models of $T$ any elementary embedding of
$M_{1}^{P}$ in $M_{2}^{P}$ extends to an elementary embedding of $M_{1}$ in $M_{2}$). It turns out that such a theorem had not been proved, although it is restated as a theorem in \cite{Hodges-book} (see Lemma 12.5.5 there).  Nevertheless that first part of my thesis did have some content, including giving a counterexample to a conjecture of Hodges.  The second part of my thesis was on problems that I  formulated myself, including the conjecture that a countable complete theory $T$ with a minimal model (no proper elementary submodel) has infinitely many countable models, up to isomorphism.  See the expository paper \cite{Pillay-Tanovic} for a discussion of work on this problem as well as on Vaught's conjecture.  

In 1975, after hearing from Gaifman about what Shelah was supposed to have proved, I asked Shelah (in the Logic Colloquium in Clermont-Ferrand) if he could send me the proof, or relevant notes. He sent me some notes in a letter,  which were not on the purported theorem mentioned by Gaifman, but on the beginnings of stability over a predicate. I could only make sense of them around 10 years later when I was asked to referee a paper by Shelah which consisted precisely of the letter he had sent in 1975.  I partly rewrote the paper and it became the joint paper \cite{Pillay-Shelah}.  In \cite{Pillay-omega-categoricity}, written in 1982,  I studied what we call here  $(\omega,\omega)$-categoricity, with a detailed structural analysis. I recently found in an old file a handwritten paper of mine on $\omega$-stability over a predicate, also from 1982.  An  improved and updated version will appear in Section 2 of the current paper.

 The paper \cite{Pillay-Shelah} was continued in Shelah's \cite{Shelah-classification-II}.  It was a kind of stream of consciousness paper and unclear to me what was actually proved. 
But according to the 2024 Master's thesis of  AlZurba (supervised by  I. Kaplan) \cite{AlZurba}, there was a structure side and a nonstructure side. The structure side is what AlZurba's  thesis was about and consisted of proving the Gaifman property under the assumption of the stability (over $P$) of $I$-systems.  The content of the nonstructure side seems to be that this  stability assumption  holds if $T$ is relatively categorical in every forcing extension of the ground model.  Even assuming this works,  Conjecture 1.3 remains open in full entirety. 
In \cite{Hart-Shelah}  for any $k$, a theory $T_{k}$ is produced which is relatively $\kappa$-categorical for all infinite $\kappa \leq \aleph_{k-1}$ but is not $2^{\aleph_{k-1}}$-categorical.  So there is no Morley's theorem for relative categoricity. The same construction produces an $L_{\omega_{1},\omega}$ sentence which is $\kappa$-categorical for all $\kappa \leq \aleph_{k-1}$ but not $2^{\aleph_{k-1}}$-categorical.  

The occasion for my coming back to the topic was seeing two papers on the subject on the arXiv in 2025, by Usvyatsov \cite{Usvyatsov} and Shelah and Usvyatsov \cite{S-U} (building on \cite{S-U-22}).  We looked at both of them in the Notre Dame model theory seminar.  The paper \cite{Usvyatsov} proved the Gaifman property under a strong stability over $P$ property (the ``stability" of {\em all} ``complete sets").   It is thematically close to what I write in Section 2.  And \cite{S-U} proves the Gaifman property under a weaker assumption (stability of ``good systems").

The rest of this section  gives the basic definitions and recalls some basic results. I also give proofs of (somewhat stronger versions of) Gaifman's syntactic characterizations  under (weak) rigidity assumptions. 
In Section 2 we study a strong form of relative categoricity, when $T$ is an ``almost internal cover" of $T^{P}$, proving Conjecture 1.3 in this case.  We also  study  $\omega$-stability or $\omega$-stability over $P$, in a rewrite of an old manuscript of mine from 1982, proving that if $T$ is (fully) $\omega$-stable over $P$, then over any complete set there is a constructible model.  We give an example of a relatively categorical $T$ which is an almost internal cover of $T^{P}$ but is not (fully) $\omega$-stable over $P$, not even fully stable over $P$ in the sense of \cite{Usvyatsov}.  The example also shows that full stability over $P$ is not preserrved when passing to $T^{eq}$.

Thanks to Julia Knight for presenting parts of \cite{S-U} in our seminar and several discussions on the subject. Thanks also to Itay Kaplan for  comments. 

We typically work in a saturated model $\bar M$ of $T$. As usual models and sets will be small subsets of $\bar M$. Truth is in the sense of $\bar M$.  $x,y,z$ denote finite tuple of variables unless said otherwise. Our assumptions from above are in place. Namely $T$ is a complete theory with quantifier elimination in a countable relational language $L$ and $P$ is a unary predicate in $L$. 
 
The following definition is from \cite{Pillay-Shelah}.  It was given there under a ``stable embeddedness" assumption, but makes sense in general as a kind of Tarski-Vaught property (relative to $P$). 

\begin{Definition} The set $A$ is said to be complete if for any formula $\phi(x)$ with parameters from $A$, if $\models \exists x(P(x)\wedge \phi(x))$ then there is $b\in P(A)$ such that $\models\phi(b)$.
\end{Definition} 

\begin{Remark} Any elementary substructure $N$ of ${\bar M}^{P}$ is complete.
\end{Remark}
\begin{proof} Let $\phi(x)$ have parameters from $N$, and $\models \exists x(P(x)\wedge \phi(x))$. As $T$ has $QE$ we may assume $\phi(x)$ is quantifier-free. 
Let $b\in P({\bar M})$ be such that $\models \phi(b)$.  But then (as $\phi$ is quantifier-free), ${\bar M}^{P}\models \phi(b)$. As $N\prec {\bar M}^{P}$, $N\models \phi(b)$.
So $\models \phi(b)$.
\end{proof} 

\begin{Lemma} For any complete countable $A$, there is $M\prec {\bar M}$ such that $A\subseteq M$ and $P(A) = P(M)$.
\end{Lemma}
\begin{proof} Let $\Sigma(x) = \{P(x)\}\cup \{x\neq a: a\in P(A)\}$.  Then $\Sigma$ is a nonprincipal partial type over $A$, so is omitted in some countable model $M\prec {\bar M}$ containing $A$. 
\end{proof}

\begin{Lemma} Any countable model $N$ of $T^{P}$ is equal to $M^{P}$ for some (countable) model of $T$. 
\end{Lemma}
\begin{proof} By Remark 1.5 and Lemma 1.6.
\end{proof}

For the record, we include the following,  which deals with the case when the underlying theory $T$ is stable (and also countable).
\begin{Lemma} Suppose $T$ is stable. Then   for any complete set $A$ there is $M\models T$ containing $A$ with $P(M) = P(A)$. In particular $T$ has the Gaifman property. 
\end{Lemma} 
\begin{proof} The point is that  under the stability (and countability)  assumptions there are locally atomic models over all sets, where $M$ is said to be locally atomic over $A$ if $A\subseteq M$ and for all tuples $b$ from $M$, $tp(b/A)$ is locally isolated, namely for  each $L$-formula $\phi(x,y)$ there is  $\psi(x)\in tp(b/A)$ such that $\models \psi(x) \to \phi(x,a)$ whenever $a\in A$ and $\models \phi(b,a)$.
This is Exercise 8.50 in \cite{Pillay-intro} where the reference given  is to Lascar's doctoral thesis.  It should also appear in Chapter IV of \cite{Shelah-classification}, where these locally isolated types are called ${\bf F}^{\ell}_{\aleph_{0}}$-types.

So let $A$ be complete and $M$ locally atomic over $A$. Assume for a contradiction that there is $b\in P(M)$, $b\notin A$. So $tp(b/A)$ satisfies $P(x) \cup \{x\neq a\}:a\in P(A)\}$. Take $\phi(x,y)$ to be $x\neq y$ , so by the local isolation of $tp(b/A)$ there is $\psi(x)\in tp(b/A)$ such that $\models \psi(x) \to x\neq a$ for all $a\in P(A)$.  But $\models\exists x(P(x)\wedge \psi(x))$. By completeness of $A$ there is $a\in P(A)$ such that $\models \psi(a)$, a contradiction. 
\end{proof}

The following is Proposition 14 of \cite{Pillay-omega-categoricity},  Proposition 3.3 of \cite{HHM}, and also Lemma 12.5.3 of \cite{Hodges-book}. 
\begin{Lemma} $T$ is  $(\omega, \omega)$-categorical if and only for every $M\models T$, $M$ is atomic over $P(M)$ (i.e. for every finite tuple $a$ from $M$, $tp_{M}(a/P(M))$ is isolated). 
\end{Lemma}
\begin{proof}  Assume the left hand side. Let $M$ be a countable model of $T$.  If $M$ is not atomic over $P(M)$, let $a$ be a finite tuple from $M$ such that $p = tp(a/P(M))$ is not isolated. By omitting types and the fact that $P(M)$ is complete we can find a countable model $M'$ such that $P(M') = P(M)$ and $M'$ omits $p$. So $M$ is not isomorphic to $M'$ over the $P$-parts, contradiction.

Now let $M$ be an arbitrary model of $T$, and $a$ a finite tuple from $M$. Let $M_{0}$ be a countable elementary substructure of  $M$ containing $a$. By what we have just seen, $tp(a/P(M_{0}))$ is isolated, by a formula $\phi(x)$ over $P(M_{0})$. It then follows that $\phi(x)$ isolates $tp(a/P(M))$. For otherwise for some $L$-formula $\psi(x,y)$ and parameters $m$ from $P(M)$, we have that both $\phi(x)\wedge \psi(x,m)$ and $\phi(x)\wedge\neg\psi(x,m)$ are consistent.
As $M_{0}\prec M$ we can find such $m\in P(M_{0})$ a contradiction.

Now assume the right hand side. Let $M_{1}$, $M_{2}$ be countable models of $T$ (elementary substructures of $\bar M$) such that $P(M_{1}) = P(M_{2}) = A$, say. By quantifier elimination (of $T$), $Th(M_{1},a)_{a\in A} = Th(M_{2}, a)_{a\in A}$. But by assumption both $M_{1}$ and $M_{2}$ are atomic in the language $L_{A}$ ($L$ with constants for elements of $A$). Hence $M_{1}$ and $M_{2}$ are isomorphic over $A$. 

\end{proof} 

The following also appears in \cite{Pillay-omega-categoricity}. The conclusion is also know as ``stable embeddability" of $P$. 

\begin{Corollary} (Uniform definability of types over the $P$-part.) Assume $T$ is  $(\omega, \omega)$-categorical. Then, for each $L$-formula $\psi(x,y)$, there is an $L$-formula $\chi(y,z)$ such that for each model $M$ of $T$ and tuple $a\in M$ there is 
$c\in P(M)$ such that for all $b\in P(M)$, $\models \psi(a,b)$ iff $\models \chi(b,c)$. 
\end{Corollary}
\begin{proof} First note that for any $M\models T$, and finite tuple $a$ from $M$, the isolation of $tp(a/P(M))$ implies its definability: Let $\phi(x)$ be a formula over $P(M)$ isolating $tp(a/P(M))$. Then for
any $L$-formula  $\psi(x,y)$ and $b\in P(M)$, $\models\psi(a,b)$ iff $\models (\forall x)(\phi(x) \to \psi(x,b))$.
Uniformity is because this holds in every model.
\end{proof} 

\begin{Remark} One can ask whether the Gaifman property follows from just stable embeddedness of $P$.   Itay Kaplan pointed out in a talk in Vienna (July 2025) a counterexample due to Hrushovski. 
Now when $T$ is stable, stable embeddability of $P$ is automatic, and in any case the existence of locally atomic models gives the Gaifman property.
In the same talk Kaplan  announced a result by  Bays, Simon, and himself, that if $T$ is  simple and $P$ is stably embedded then $T$ has the Gaifman property. 
\end{Remark}  

Let us mention an easy extension of Lemma 1.7 to the case where the model $N$ of $T^{P}$ has cardinality at most $\aleph_{1}$ under an $(\omega, \omega)$-categoricity assumption. First:

\begin{Lemma}  Suppose $T$ is  $(\omega, \omega)$-categorical. Let $M$ be a model of $T$, and let $N_{1}\prec N_{2}$ be models of $T^{P}$ such that $N_{1} = M^{P}$.
 Then $M\cup N_{2}$ is complete.
\end{Lemma}
\begin{proof} So everybody is living inside ${\bar M}$.  Suppose that $\phi(x,y,z)$ is an $L$-formula, $b\in M$, $c\in N_{2}$ and  ${\bar M}\models \exists x(P(x)\wedge \phi(x,b,c))$. 
By Lemma 1.8,  $tp(b/P(M))$ is isolated, by a formula $\chi(y,d)$ with $d\in P(M)$.  Now as discussed earlier $\chi(y,d)$ isolates $tp(b/P({\bar M}))$.
Let $e\in P({\bar M})$ be such that  ${\bar M}\models \phi(e,b,c)$. So  $ {\bar M}
\models (\forall y)(\chi(y,d) \to \phi(e,y, c))$.   Now consider the formula $(\forall y)(\chi(y,d)  \to \phi(x,y,c))$. By QE it is equivalent (in ${\bar M})$ to a quantifier-free formula 
$\pi(x,d,c)$. Now $d,c$ are parameters from $N_{2}$ and $\pi(x,d,c)$ is realized in ${\bar M}^{P}$, so as $N_{2}\prec {\bar M}^{P}$, it is realized by some $e'\in N_{2}$. 
So ${\bar M}\models \pi(e',d,c)$.  Hence ${\bar M}\models (\forall y)(\chi(y,d) \to \phi(e',y,c))$. So  ${\bar M}\models \phi(e',b,c)$ as required. 
\end{proof}

\begin{Proposition} Suppose $T$ is $(\omega, \omega)$-categorical and $N\models T^{P}$ has cardinality $\leq \aleph_{1}$. Then $N = M^{P}$ for $M$ a model of $T$. 
\end{Proposition}
\begin{proof}
When $N$ is countable, this is Lemma 1.7 (which needs no assumption of $(\omega, \omega)$-categoricity).
So suppose $N$ has cardinality $\aleph_{1}$.  Again we work in ${\bar M}$. So $N\prec {\bar M}^{P}$.
Write $N$ as the union of a continuous elementary chain $(N_{\alpha}: \alpha < \aleph_{1})$ of countable models (of $T^{P}$). 
We will build a a continuous  elementary chain $(M_{\alpha}:\alpha < \aleph_{1})$ of countable models of $T$ (elementary substructures of ${\bar M}$) such that 
for each $\alpha$, $N_{\alpha} = M_{\alpha}^{P}$.  $M_{0}$ is given by Lemma 1.7.
Suppose we have found $M_{\alpha}$. Then by Lemma 1.11,  $M_{\alpha}\cup N_{\alpha + 1}$ is complete, so by Lemma 1.6,  we find our countable model $M_{\alpha+ 1}$ of $T$ containing $M_{\alpha}$ and with $N_{\alpha + 1} = M_{\alpha}^{P}$.   So we can build the $M_{\alpha}$.  Let $M = \cup_{\alpha}M_{\alpha}$.
\end{proof}

\begin{Remark}  (i) The reader can easily see that the conclusion of Lemma  1.11 just follows from stable embeddedness (definability of types over the $P$-part) by the same proof. As the proof of Proposition 1.13 just uses the conclusion of Lemma 1.12 it follows only from stable embeddedness  (also mentioned in Kaplan's Vienna talk). 

(ii) One is tempted to continue the construction above beyond $\aleph_{1}$, by for example trying to first show (using a similar union of chains argument) that if $M, N_{1}, N_{2}$ are as in Lemma 1.12, and of  cardinality $\aleph_{1}$ then there is $M'$ containing $M$ with $N_{2} = M'^{P}$. However this would require us to know the completeness of more complicated configurations, which are ``good systems" in the sense of \cite{S-U}. 
\end{Remark}

Recall that we call $(T,P)$ $1$-cardinal if for any model of $T$, $|M| = |P(M)|$,  which is well-known to be equivalent to there being no uncountable $M$ for which $P(M)$ is countable, and also equivalent to  there being no elementary pair $M_{1}\prec M_{2}$ of models with $M_{1}\neq M_{2}$ and $P(M_{1}) = P(M_{2})$
The following is also Proposition 3.2 of \cite{HHM} and Theorem 12.5.4 of \cite{Hodges-book}. 
\begin{Lemma} The following are equivalent:
\newline
(a) $T$ is relatively $\omega$-categorical.
\newline
(b) $T$ is $(\omega, \omega)$-categorical, and $(T,P)$ is $1$-cardinal.
\end{Lemma}
\begin{proof} Immediate or left to the reader. 
\end{proof}

It is convenient at this point to give quick proofs of the equivalences   stated by  Gaifman in \cite{Gaifman} and \cite{Gaifman-conjecture}, but in a slightly stronger form. 

\begin{Lemma} The following are equivalent:
\newline
(i) $T$ is $(\omega, \omega)$-categorical and every countable model $M$ of $T$ is ``rigid" over $P(M)$, namely $Aut(M/P(M))$ is trivial.
\newline
(ii) For every model $M$ of $T$, $M = dcl(P(M))$. 
\newline
(iii) $T$ is relatively categorical and for every $M\models T$, $M$ is rigid over $P(M)$. 
\end{Lemma}
\begin{proof}   It is easy to see that (ii) implies (iii)  and (iii) implies (i). 
\newline
For (i) implies (ii) we first show that (ii) holds for countable models.  Let $M$ be countable, and suppose for a contradiction that $M\neq dcl(P(M))$.  So there is an element $a\in M$ not in $dcl(P(M))$. But $tp(a/P(M))$ is isolated, by  a formula $\phi(x)$ over $M$, which must then have another realization $b$ in $M$. As $M$ is atomic, so $\omega$-homogeneous, over $P(M)$, there is an automorphism $f$ of $M$ over $P(M)$ taking $a$ to $b$, contradicting rigidity. 
Now if $M$ is an arbitrary model of $T$, and $a\in M$, let $M_{0}$ be a countable elementary substructure of $M$ containing $a$. Then by what we just proved $a\in dcl(P(M_{0}))$ so $a\in dcl(P(M))$. 
\end{proof} 

We now recall internality. 
\begin{Definition} Let ${\bar M}$ be the saturated model of $T$. We say  that ${\bar M}$ is internal to $P$ if there is a small subset $A$ of ${\bar M}$ such that ${\bar M} = dcl(A,P({\bar M}))$.
Following \cite{Hrushovski} we also say that $T$ is an internal cover of $T^{P}$
\end{Definition} 

\begin{Remark}
(i) If  $T$ is an internal cover of $T^{P}$ then by compactness there is a finite tuple ${\bar a}$ from  ${\bar M}$ and $\emptyset$-definable partial function  $f(-,-)$ such that every element of ${\bar M}$ is of the form $f({\bar a}, {\bar d})$ for some tuple ${\bar d}$ from $P(M)$. 
\newline
(ii) It follows that for every model $M$ of $T$ there is a finite tuple $a$ from $M$ such that every element of $M$ is of the form $f(a,d)$ for some $d\in P(M))$. 
\end{Remark}

\begin{Lemma} The following are equivalent:
\newline
(i) $T$ is $(\omega, \omega)$-categorical, and every countable model $M$ of $T$ is ``weakly rigid" over $P(M)$, namely $|Aut(M/P(M))| < 2^{\omega}$.
\newline
(ii)  $T$ is $(\omega, \omega)$-categorical and is an internal cover of $T^{P}$.
\newline
(iii) $T$ is relatively categorical and every model $M$ of $T$ is ``weakly rigid" over $P(M)$, that is $|Aut(M/P(M))| < 2^{|M|}$. 
\end{Lemma}
\begin{proof} (iii) implies (i) is clear. 
\newline
(i) implies (ii).  We first note that for any countable model $M$ of $T$, there is a finite tuple $a$ from $M$ such that $M = dcl(P(M), a)$.  This is because as is well-known it follows from $Aut(M/P(M))$ having cardinality $< 2^{\omega}$ that there is a finite tuple $a$ from $M$ such that  if $\sigma\in Aut(M/P(M))$ fixes $a$ then $\sigma$ is the identity. But then by $\omega$-homogeneity (over $P(M))$ of $M$, it follows that $M= dcl(P(M),a)$.
Now a downward Lowenheim-Skolem argument implies that  for any model $M$ of $T$,  $M= dcl(a,P(M))$ for some tuple from $M$. So we have internality to $P$. 
\newline
(ii) implies (iii).  First given $M\models T$ and $a$, $f(-,-)$ as in Remark 1.18, any $\sigma\in Aut(M/P(M))$ is determined by $\sigma(a)$, so $Aut(M/P(M))$ has  cardinality at most $|M|$ so we get weak rigidity.
\newline
Now for relative categoricity: Let $M_{1}, M_{2}$ be models of $T$ with $M_{1}^{P} = M_{2}^{P} = N$ say. 
Let $a$ be a finite tuple from $M_{1}$ as in Remark 1.18 (ii). By Lemma 1.9, Let $tp(a/N)$ be isolated by $\phi(x)$. Then As $M_{1}$ and $M_{2}$ are elementarily equivalent over $P(M_{1} = P(M_{2})$, let $b\in M_{2}$ realize $\phi(x)$.
So $tp(a/N) = tp(b/N)$ and every element of $M_{2}$ is of the form $f(b,d)$ for some $d\in N$. Then $M_{1}$ is isomorphic to $M_{2}$ over $N$ by taking $a$ to $b$ and $f(a,d)$ to $f(b,d)$ for all $d\in N$.

\end{proof}

There are several equivalent characterizations of $(T,P)$ being $1$-cardinal.  In my thesis I made use of a certain rank introduced by R. Deissler in his 1974 Ph.D. thesis (University of Freiburg).  Another is by Erimbetov in a 1974 paper in Algebra and Logic.  But it is convenient to use a nice characterization in \cite{Herwig-Hrushovski-Macpherson} in terms of ``co-analyzability". We quote their definition (although it could be simplified).
\begin{Definition} We fix our theory $T$ and predicate $P$.  And work in the saturated model ${\bar M}$. Let $Q$ be any definable (with parameters) set. 
We say that $Q$is $0$-coanalyzable in $P$ (or co-analyzable in $P$ in $0$ steps) if $Q$ is finite.
And for $k\geq 0$, $Q$ is $k+1$-co-analyzable in $P$ (or co-analyzable in $P$ in $k+1$-steps)  if there is a definable (with parameters) relation $R\subseteq  Q\times P({\bar M})^{\ell}$ for some $\ell$ such that projection of $R$ to $Q$ is surjective and for any $\ell$-tuple $c$ from $P$,  $\{x\in Q: R(x,c)\}$ is $k$-co-analyzable in $P$.
\end{Definition} 

Note that if $Q$ is $k$-co-analyzable in $P$ then this happens uniformly in the obvious sense, by saturation of ${\bar M}$.

From Proposition 2.4 of \cite{Herwig-Hrushovski-Macpherson} we have:
\begin{Fact} $(T,P)$ is $1$-cardinal iff $x=x$ is  $k$-co-analyzable in $P$ for some $k$. 
\end{Fact} 

To compare co-analyzability to analyzability recall that $Q$ is $1$-analyzable in  $P$  (or internal to $P$) if for some small (in fact finite) set $A$ of parameters, $Q\subseteq dcl(P({\bar M}), A)$,  and $Q$ is $k+1$-analyzable in $P$ if there is a definable (with parameters) map $f:Q\to Q'$ such that $Q'$ is $k$-analyzable in $P$ and evetry fibre is $1$-analyzable in $P$. 

Recall that $Q$ is said to be ``almost internal to $P$" if for some finite set of parameters $A$, $Q\subseteq acl(P({\bar M}), A)$.  So this is like internality, except $dcl$ is weakened to $acl$. 

Of course we are interested in the case where $Q$ is the universe of ${\bar M}$, namely the solution set of $x=x$, in which case we will say for example that ${\bar M}$ is $k$-co-analyzable in $P$. 

\begin{Lemma} The following are equivalent:
\newline
(i) ${\bar M}$ is $1$-co-analyzable in $P$,
\newline
(ii) ${\bar M}$ is almost internal to $P$,
\newline
(iii)  There is an $L$-formula $\phi(x,y,z)$ and $k<\omega$ such that $T$ implies ``for all $y$ and all $z$ from $P$, there are at most $k$ x such that $\phi(x,y,z)$ holds", AND for any model $M$ of $T$ there is a tuple $a$ from $M$ such that  $M\models (\forall x)(\exists z)(z\in P \wedge \phi(x,a,z))$.

\end{Lemma}
\begin{proof} The equivalence of (i) and (ii) is from the definitions.  The equivalence with (iii) is precisely the uniformity discussed above.
\end{proof}

As in Definition 1.17 we will say that $T$ is an almost internal cover of $T^{P}$  when the equivalence conditions in Lemma 1.22 hold.

\section{Almost internal covers and $\omega$-stability over a predicate}
We know by Lemma 1.15  and Fact 1.21 that if $T$ is relatively $\omega$-categorical then for some $k$, $x=x$ is $k$-co-analyzable in $P$.   By \cite{Hart-Shelah} there are relatively $\omega$-categorical theories which are not relatively  categorical. It would be interesting to look for  a strengthening of co-analyzability so as to obtain an equivalence with relative categoricity (modulo relative $\omega$-categoricity).  It would also be interesting to look at the examples from \cite{Hart-Shelah} through the lens of co-analyzability.  In the meantime we just point out that $1$-co-analyzability {\em is} a sufficient condition for relative categoricity and moreover the Gaifman conjecture holds in this case.  
As above we will say that $T$ is an almost internal cover of $T^{P}$ when $x=x$ is $1$-co-analyzable in $P$. 

So we will prove:
\begin{Proposition} Suppose $T$ is relatively categorical and ${\bar M}$ is almost internal to $P$. Then 
\newline
(i) $T$ is relatively categorical, and
\newline
(ii) $T$ has the Gaifman property.
\end{Proposition} 

Actually (i) follows as in the proof of Lemma 1.19.   One can prove (ii) by an explicit construction which we discuss later. But it will be amusing to first prove it via showing that the unproved ``result" of Shelah that I appealed to in my thesis is true in this case. 

\begin{Lemma} Suppose $T$ is relatively  categorical and $T$ is an almost internal cover of $T^{P}$. Let $M_{1}, M_{2}$ be models of $T$, and let $N_{i} = M_{i}^{P}$ for $i=1,2$. Suppose $f:N_{1}\to N_{2}$ is an elementary embedding, then $f$ extends to an elementary embedding of $M_{1}$ in $M_{2}$. 
\end{Lemma} 
\begin{proof} We may assume that $N_{1}\prec N_{2}$. We again let $\phi(x,y,z)$ be the $L$-formula as in (iii) of Lemma 1.22.  And let $a$ be a finite tuple from $M_{1}$ such that $M_{1}$ is the set of realizations of
$\phi(x,a,b)$ as $b$ ranges over tuples from $P(M_{1}) = N_{1}$.  Let $\chi(y,d)$ isolate $tp(a/P(M_{1}))$ with $d\in P(M_{1})$.  So easily in $M_{1}$ the following holds: for all $a'$ realizing $\chi(y,d)$, $M_{1}$ is the set of realizations of the $\phi(x,a',b)$ as $b$ ranges over tuples from $P(M_{1})$.  (First check it holds in the saturated model ${\bar M}$, as there will be an automorphism fixing $P({\bar M})$ pointwise and taking $a$ to $a'$.  
So it holds in $M_{2}$.)  Let $a'$ realize $\chi(x,d)$ in $M_{2}$.  So $tp(a'/P(M_{1}) = tp(a/P(M_{1}))$.  Consider the map from $M_{1}$ to $M_{2}$ which is the identity on $N_{1}$, takes $a$ to $a'$, and takes  $M_{1} = acl(N_{1},a)$ to $acl(N_{1},a')\subseteq N_{2}$.  This is the required  elementary embedding. 

\end{proof}

\noindent
We now  proceed as in my thesis. 
\newline
{\em Proof of Proposition 2.1 (ii).} 
 We prove by induction on the cardinality of $N\models T^{P}$ that $N = M^{P}$ for some model $M$ of $T$.
When $T$ is countable this is Lemma 1.7.
Now assume $\kappa$ to be uncountable and that every model of $T^{P}$ of cardinality $<\kappa$ is the $P$-part of a model of $T$.
Let $N$ be a model of $T^{P}$ of cardinality $\kappa$ and let $(N_{\alpha}: \alpha < \lambda)$ be a continuous chain of elementary submodels of $N$, each of cardinality $< \kappa$ and whose union is $N$.
By the inductive assumption, for each $\alpha < \kappa$, let $M_{\alpha}$ be a model of $T$ such that $N_{\alpha} = M_{\alpha}^{P}$, and notice it has the same cardinality as $N_{\alpha}$. 
By Lemma 2.2 we may assume that $M_{\alpha}\prec M_{\alpha + 1}$ for all $\alpha$.
Moreover for a limit ordinal $\delta < \lambda$ we may assume that $\cup_{\alpha<\delta}M_{\alpha} = M_{\delta}$,  because as these are both models of $T$ with the same $P$ parts they are isomorphic over the $P$-parts.
Put $M = \cup_{\alpha< \lambda}M_{\alpha}$ and $M$ is a model of $T$ with $M^{P} = N$.   \qed

\begin{Remark} The explicit construction of $M$ such that $N = M^{P}$ goes as follows:  Choose countable $N_{0}\prec N$. Let  by 1.7 $M_{0}$ be a model of $T$ such that $N_{0} = M_{0}^{P}$. And let $\phi(x,y,z)\in L$ be as in Lemma 1.22 and let tuple $a$  be in $M_{0}$ such that $M_{0}$ is the set of realizations of formulas $\phi(x,a,b)$ as $b$ ranges over $N_{0}$.  So also ${\bar M}$ is the set of realizations of $\phi(x, a, b)$ as $b$ ranges over ${\bar M}^{P}$. In particular ${\bar M} = acl(P({\bar M}),a)$.  Now show that $acl(N,a)$ is an elementary substructure of ${\bar M}$ whose $P$-part is exactly $N$.  For example to show the latter, we use the fact that $tp(a/P({\bar M}))$ is definable over $N$ to see that $acl(N,a)\cap P({\bar M}) = acl(N)\cap P({\bar M})$. As $N\prec {\bar M}^{P}$ this must be exactly $N$. 

The same construction will work only assuming that $P$ is stably embedded (and ${\bar M}$ is almost internal to $P$), yielding that these assumptions also  imply the Gaifman property. 

\end{Remark}

\vspace{5mm}
\noindent
We now pass on to stability.  Our general assumptions are as at the beginning of Section 1: $T$ a complete theory in a countable relational language $L$, $T$  has QE, and $P\in L$ a unary predicate.

\begin{Definition}
Let $A$ be a complete set (as in Definition 1.4) and let $p(x)$ be a complete type over $A$ ($x$ a finite tuple of variables). $p(x)$ is said to be good if for some/any realization $a$ of $p$, $(A,a)$ is complete.
$S_{*}(A)$ is the collection of good complete types over $A$.

\end{Definition}  

\begin{Remark} If $A$ is complete and $p(x)\in S(A)$ is isolated then  $p$ is good.
\end{Remark} 
\begin{proof} Suppose $p$ is isolated by the formula $\phi(x)$ over $A$.  Let $a$ realize $p$. Let $\psi(x,y)$ be a formula over $A$ such that $\models \exists y(P(y) \wedge (\psi(a,y))$. Then 
$\models \exists y(P(y) \wedge \exists x(\phi(x)\wedge \psi(x,y)))$.  By completeness of $A$ there is $b\in P(A)$ such that $ \exists x(\phi(x)\wedge \psi(x,b))$. 
As $\phi(x)$ isolates $tp(a/A)$ and $b\in A$ we have $\models \psi(a,b)$, as required. 
\end{proof}

The general thrust of ``stability over a predicate" is to count good types over complete sets, in the same way as stability can be introduced by counting types over  arbitrary sets. (However these sets $S_{*}(A)$ of good types are NOT compact spaces.)   In \cite{Pillay-Shelah} (and  later papers)   there was a blanket assumption of stable embeddabilty of $P$.  A complete set $A$ was called {\em stable}  (or stable over $P$) if  $|S_{*}(A)| \leq |A|^{\omega}$  (assuming $T$ countable).   The stability over $P$ over $A$ was shown to be equivalent to certain local ranks on types over $A$ being finite.  Some consequences of the instability over $P$ of some model of $T$ were given in  terms the existence of ``many" models of $T$ in suitable cardinalities with the same $P$-part, and under some set-theoretic hypotheses.  In \cite{Shelah-classification-II} (and \cite{AlZurba})  the stability of suitable ``systems" was assumed, in the presence of a $1$-cardinality assumption, to show the Gaifman property.  According to \cite{Usvyatsov}, the more recent work (by Usvyatsov, Shelah) was aimed at proving the Gaifman property,  under assumptions on the stability (over $P$) of suitable systems, but without assuming $1$-cardinality.

In \cite{Usvyatsov}, $T$  is called {\em fully stable} over $P$ if {\em every}  complete set is stable over $P$.   Under this assumption, so-called ''full existence" was proved, namely for every complete set $A$ there is a model $M$ of $T$ such that $P(M) = P(A)$, by building a ``locally atomic" model over $A$.    It was also pointed out that taking $T$ to be some completion of $ACFA_{0}$ and $P$ the fixed field, $T$ is fully stable over $P$ (although $T$ is far from being relatively categorical).

Bearing in mind the role of $\omega$-stability in the study of uncountable categoricity, one might expect that some version of $\omega$-stability over $P$ would play a role in relative categoricity.  So we introduce (full)  $\omega$-stability over $P$ and make some observations, although we see later some very basic examples of relatively categorical theories which are not fully $\omega$-stable over $P$.

\begin{Definition} We will say that $T$ is fully $\omega$-stable over $P$, or fully  relatively $\omega$-stable, if for every countable complete set $A$, $S_{*}(A)$ is countable. 
\end{Definition} 

\begin{Remark} (i)  We are using the expression ``fully" to be consistent with \cite{Usvyatsov}. Instead we might want to just posit that for suitable countable complete sets $A$, $S_{*}(A)$ is countable. 
\newline
(ii)  Full $\omega$-stability will imply that for every complete set $A$, $|S_{*}(A)| \leq |A|$ (although we will not prove it).  In particular it will imply that $T$ is fully stable over $P$ in the sense of \cite{Usvyatsov}. 
\newline
(iii)  Suppose $T$ is relatively $\omega$-categorical and $M$ is a countable model of $T$ (so complete), then $S_{*}(M)$ is precisely the set of realized types.
\newline
(iv) Suppose $T$ is  $(\omega, \omega)$-categorical, and $N$ is a model of $T^{P}$ so  complete by Remark 1.5.  Then $S_{*}(N)$ is countable. 
\end{Remark} 
\begin{proof}
(iii) Let $M\models T$ be countable,  $p(x)\in S_{*}(M)$ and $a$ realize $p$. Then  by Lemma 1.6, there is a countable  model $M'$ of $T$ containing $M,a$ with $P(M') = P(A) =  P(M)$.  As $M$ has no proper elementary extension with the same $P$-part (by $1$-cardinality), $M' = M$ and $p$ is realized in $M$.
\newline
(iv)  By Lemma 1.7, there is a countable model $M$ of $T$ such that $M^{P} = N$, and by $(\omega, \omega)$-categoricity $M$ is unique up to isomorphism over $N$.  If $tp(a/N)$ is good, then again by Lemma 1.6, $Na$ extends to a model $M_{a}$ of $T$ with $M_{a}^{P} = M^{P} = N$. But then $M_{a}$ is isomorphic to $M$ over $N$ so in particular $tp(a/N)$ is realized in $M$. So $S_{*}(N)$ is countable.

\end{proof}

The main thing is: 
\begin{Proposition} Suppose that $T$ is fully $\omega$-stable over $P$. Let $A$ be complete. Then the isolated types are dense in $S(A)$: for any consistent formula $\phi(x)$ over $A$, there is an isolated $p(x)\in S(A)$ containing $\phi(x)$.
\end{Proposition}
\begin{proof} We will assume the conclusion is false and use a routine argument to build continuum many good types over a countable complete subset of $A$, contradicting the $\omega$-stability assumption. 
So suppose $\phi(x)$ is a consistent formula over $A$ which is not in any isolated complete type over $A$.  We fix a bijection $f: \omega\setminus \{0\} \to \omega\times \omega$ such that if $f(n) = (r,i)$ then $r < n$. 

We build inductively formulas  $\psi_{\tau}(x)$ over $A$ for $\tau\in 2^{<\omega}$, countable complete subsets $A_{n}\subseteq A$ for $n<\omega$, and for each $n<\omega$ a list $\{\sigma_{(n,i)}(x): i<\omega\}$ of all formulas over $A_{n}$ of the form $\exists y \in P(\chi(x,y))$ with the properties::
\newline
(i)  $\psi_{<>}(x) = \phi(x)$, each $\psi_{\tau}(x)$ is consistent,   $\tau$ an initial segment of $\tau'$ implies $\models \psi_{\tau '}(x) \to \psi_{\tau}(x)$, and for each $\tau$ if $\tau_{1}$, $\tau_{2}$ are the $2$ successors of $\tau$ then $\psi_{\tau_{1}}(x)$ and $\psi_{\tau_{2}}(x)$ are contradictory,
\newline
(ii)  The $A_{n}$ are increasing and for each $n$ all the formulas $\psi_{\tau}(x)$ for $\tau$ of length $< n$  have parameters from $A_{n}$.
\newline
(iii)  Suppose $\ell(\tau) = n\geq 1$, and let $\sigma_{f(n)}(x)$ be $\exists y\in P (\chi(x,y))$. Then either $\models \psi_{\tau}(x) \to \neg\sigma_{f(n)}(x)$ or for some $c\in P(A_{n})$, $\models \psi_{\tau}(x) \to \chi(x,c)$. 

\vspace{2mm}
\noindent
To begin, let (by a downward Lowenheim-Skolem argument) $A_{0}$ be a countable complete subset of $A$ containing the parameters from $\phi(x)$. 
Suppose now $n< \omega$,  and $A_{r}$, $\psi_{\tau}(x)$ for $\ell(\tau) \leq n$ and also the list $\sigma_{r,i}(x)$ for $r\leq n$ and $i<\omega$, have been defined. 
Now for each $\tau$ of length $n$, $\psi_{\tau}(x)$ does not isolate a complete type over $A$ so there is a formula $\psi_{\tau}'(x)$ over $A$ such that both $\psi_{\tau}(x)\wedge \psi_{\tau}'(x)$ and $\psi_{\tau}(x)\wedge \neg\psi_{\tau}'(x)$ are consistent.   Let $\psi_{\tau ,0}''(x)  = \psi_{\tau}(x)\wedge \psi_{\tau}'(x)$ and $\psi_{\tau, 1}''(x) = \psi_{\tau}(x)\wedge \neg\psi_{\tau}'(x)$. 
So we have now defined $\psi_{\tau}''(x)$ for all $\tau$ of length $n+1$.  Now for $\tau$ of length $n+1$, if $\psi_{\tau}''(x)\wedge \sigma_{f(n+1)}(x)$ is consistent and $\sigma_{f(n+1}(x)$ is $(\exists y\in P)(\chi(x,y))$ then 
$(\exists y\in P)(\exists x)(\psi_{\tau}''(x) \wedge \chi(x,y))$ is over $A$ and consistent, so by completeness of $A$ pick $c_{\tau} \in P(A)$ such that $\psi_{\tau}''(x)\wedge \chi(x,c)$ is consistent.  Put $\psi_{\tau}(x) = \psi_{\tau}''(x)\wedge \chi(x,c)$.  Otherwise $\psi_{\tau}''(x)$ implies $\neg\sigma_{f(n+1)}(x)$, and just put $\psi_{\tau}(x) = \psi_{\tau}''(x)$. 

Now let $A_{n}'$ be $A_{n}$ together with the parameters in the formulas $\psi_{\tau}'(x)$ for $\tau$ of length $n$, and the $c_{\tau}$ for $\tau$ of length $n+1$. Let again by a  Lowenheim-Skolem argument, $A_{n+1}$ be a countable complete subset of $A$ containing $A_{n}'$.  Finally let  $\{\sigma_{n+1, i}(x): i<\omega\}$ be a list of all formulas over $A_{n+1}$ of the appropriate form $\exists y\in P(\chi(x,y))$.

\vspace{2mm}
\noindent
So we can carry out the inductive construction. Let $A'$ be the union of the $A_{n}$. So $A'$ is countable and complete. For each  $\eta\in 2^{\omega}$ let $p_{\eta}(x) = \{\psi_{\eta|n}(x): n <\omega\}$. So by (i), the $p_{\eta}$ are consistent sets of formulas over $A'$, and if $\eta \neq \eta'$, $p_{\eta}(x)$ and $p_{\eta'}(x)$ are contradictory.
 For each $\eta\in 2^{\omega}$ extend $p_{\eta}(x)$ to a complete type $p_{\eta}'(x)$ over $A'$.
\newline
{\em Claim.}  For $\eta\in 2^{\omega}$, $p_{\eta}'(x) \in S_{*}(A')$ namely is a good type over $A'$. 
\newline
{\em Proof of Claim.}  We must show that whenever a formula of the form $(\exists y\in P)(\chi(x,y))$ is in $p_{\eta}(x)$ then for some $c\in P(A')$, $\chi(x,c)\in p_{\eta}(x)$.  Now the formula $(\exists y\in P)(\chi(x,y))$ is over $A'$ so 
over $A_{n}$ for some $n$, so of the form $\sigma_{n,i}(x)$ for some $i$. Let $1\leq m < n$ be such that $f(m) = (n,i)$. Then $\psi_{\eta|m}(x) \in p_{\eta}'(x)$. So we cannot have that $\psi_{\eta|m}(x)$ implies $\neg\sigma_{n,i}(x)$.  So by (iii) 
of the construction $\models \psi_{\eta|m}(x) \to \chi(x,c)$ for some $c\in P(A_{m})\subset P(A')$. So $\chi(x,c) \in p_{\eta}'(x)$ and the Claim is proved.  \qed

So by the Claim and as the $p_{\eta}'(x)$ are distinct complete types over $A'$, we have that $S_{*}(A')$ is uncountable, contradicting the full $\omega$-stability over $P$ assumption.
This proves Proposition 2.8. 

\end{proof} 

\begin{Corollary} Suppose $T$ is fully $\omega$-stable over $P$. Then over any complete set $A$ there is a constructible (so prime) model $M_{A}$ with $P(M_{A}) = P(A)$.
\end{Corollary}
\begin{proof} By iterating Remark 2.5 and Proposition 2.8 we can find $M\prec {\bar M}$ with $M = A\cup\{a_{\alpha}: \alpha < \kappa\}$ where for each $\alpha$, $tp(a_{\alpha}/A\cup \{a_{\beta}:\beta < \alpha\})$ is isolated, and each $A_{\alpha} = A\cup \{a_{\beta}: \beta <\alpha\}$ is complete with $P(A_{\alpha}) = P(A)$.

\end{proof}

\begin{Proposition} Suppose $(T,P)$ is $1$-cardinal, and $T$ is fully $\omega$-stable over $P$. Then $T$ is relatively categorical and has the Gaifman property.
\end{Proposition}
\begin{proof} Let $M_{1}$, $M_{2}$ be models of $T$ with $M_{1}^{P} = M_{2}^{P} = N$.  As $N$ is complete, let $M_{N}$ be the constructible model over $N$ given by Corollary 2.9. So $P(M_{N}) =N$, and $M_{N}$ elementarily embeds in each of $M_{1}, M_{2}$ over $N$. By $1$-cardinality these elementary embeddings are isomorphisms, so $M_{1}$ is isomorphic to $M_{2}$ over $N$. 

The Gaifman property is  the special case of Corollary 2.9 when $A$ is a model of $T^{P}$.

\end{proof}

In particular from Proposition 2.9 we conclude that if $T$ is relatively $\omega$-categorical and fully $\omega$-stable over $P$, then $T$ is relatively categorical. Hence the examples in \cite{Hart-Shelah} which are relatively $\omega$-categorical but not relatively $\kappa$-categorical for some $\kappa > \omega$ are not fully $\omega$-stable over $P$. 
When  $T$ is as in Lemma 1.16 ($M = dcl(P(M))$ for all $M\models T$), then $T$ is fully $\omega$-stable over $P$:  Suppose $A$ is countable and complete. By Lemma 1.6 let $M$ be (countable) model of $T$ containing $A$ with $P(M) = P(A)$. So $A\subseteq dcl(P(M))$. So any good type over $A$ corresponds to a good type over $P(M)$, which must be realized in $M$. As $M$ is countable there can be only countably many such good types.
A similar argument works when $M = acl(P(M))$ for all $M\models T$.

The next level in ``complexity" in a sense is when $T$ is an (almost) internal cover of $T^{P}$, where one needs the parameter set $A$ outside $P(M)$  such that $M = dcl(A\cup P(M))$  (or with $acl$).
The most basic version is $n$-dimensional vector spaces over a field.  The (incomplete) theory of an $n$-dimensional vector space $V$ over a field $F$, consists of two sorts, one for a field  (equipped with its field structure) and one for an $n$-dimensional vector space $V$ over $F$ with its addition, zero and scalar multiplication $\alpha: F\times V \to V$.  With the obvious theory. So $P$ will stand for the field sort and the complement of $P$ for the vector space sort. Call this 
incomplete theory $T_{n}$.  Of course $T_{n}$ is relatively categorical, in the obvious sense. 
\begin{Lemma} Let $T$ be a relatively categorical expansion of $T_{n}$. Then $T$ is relatively $\omega$-stable.

\end{Lemma}
\begin{proof} Let $A$ be a countable complete set. Then by  Lemma 1.6 there is a countable model $M$ containing $A$ with $P(M) = P(A)$. We can write $M$ as an expansion of $(V,F)$ where $F = P(M)$ and $V$ is the complement.
Hence $A = F \cup B$ where $B\subseteq V$.  Let ${\bar b}$ be a maximal $F$-linearly independent subset of $B$, so ${\bar b}$ is finite and $B\subseteq dcl(F,{\bar b})$.  Let $p(x)$ be a good complete type over $B$, so $p(x)$ can be rewritten as a good complete type over $F\cup\{\bar b\}$.  Let $a$ realize $p(x)$.  So $tp({\bar b},a/F)$ is realized in $M$. In particular, as $M$ is atomic over $P(M)$. $tp({\bar b},a)/F)$ is isolated, so it follows that $p(x)$ is isolated. So there can be only countable many good types over $A$.

\end{proof}

We will end with an example of a relatively categorical theory $T$ which is an internal cover of $T^{P}$, but is not fully $\omega$-stable, and not even fully stable over $P$.   

 First the syntax: we have unary predicates  $P, G, I, X, Y$ and the theory says that these are disjoint, $P$ is the union of $G$ and $I$, and $\neg P$ the union of $X$ and $Y$. $+$ is a commutative group operation on $G$. 
 We have $F:G\times X \to X$, and the theory says this is a strictly transitive group action.  $R$ is a binary predicate symbol and the theory says $R\subseteq G\times I$ and for $i\in I$, $R(x,i)$ defines a subgroup of $G$ of index $2$ (which we call $G_{i}$).   

We also have a function symbol $h$ and the theory says that $h$ is a surjective function from $Y$ to $I$, and that for each $i\in I$, $h^{-1}(i)$ (which we call $Y_{i}$) has cardinality $2$.  We have another function symbol which identifies  $Y_{i}$ with the  set of orbits under $G_{i}$ in $X$. Specifically we have a surjective function $j:X\times I \to Y$, and axioms saying that  $j(x,i) = j(y,i)$ iff $x,y$ are in the same orbit under $G_{i}$, and for a given $i\in I$, the image of $X$ under $j(-,i)$ is precisely $Y_{i}$. 

Call this language $L$.  We describe an $L$-structure $M$ satisfying all the axioms above. Let $G(M)$ be $(\Z/2\Z)^{(\omega)}$  (direct sum). Let  $I(M) = \omega$.  Let $G_{i}(M)$ be  those elements of $G(M)$ with $ith$-coordinate 
$0$.  Let $X(M)$ be a principal homogeneous space for $G(M)$, where $F(M)$ is the action.  Let $Y_{i}(M)$ be the orbits of $G_{i}(M)$ on $X$ and the interpretations of $h$ and $j$ in $M$ the obvious things. 
Let $P(M)$ be the disjoint union of $G(M)$ and $I(M)$.
Let $T = Th(M)$ in the language $L$. 

For the record, we let $M_{0}$ be the relativized reduct of $M$ where we ignore the $Y$  sort and the symbols $h$ and $j$. 
Let $T_{0} = Th(M_{0})$ in the appropriate language $L_{0}$. 

\begin{Remark} $M$ is contained in $M_{0}^{eq}$.
\end{Remark}
\begin{proof} In $M$ we are just naming the orbits under the action of $G_{i}(M_{0})$ on $X(M_{0})$. 
\end{proof}

\begin{Proposition} (i) $T_{0}$ and $T$ are relatively categorical (with respect to $P$).
\newline
(ii) $T_{0}$ and $T$ are internal covers of $T_{0}^{P}$, $T^{P}$ respectively.
\newline
(iii) $T_{0}$ is (fully) $\omega$-stable over $P$.
\newline
(iv) $T$ is not (fully) $\omega$-stable over $P$. Moreover for any infinite cardinal $\kappa$ there is a complete set $A_{\kappa}$ of cardinality $\kappa$ such that  $|S_{*}(A_{\kappa})| = 2^{\kappa}$.  So $T$ is not fully stable over $P$ (in a strong way).  
\end{Proposition} 
\begin{proof} (i) By Remark 2.12 it is enough  to prove it for $T_{0}$. But it is immediate: if $M_{1}, M_{2}$ are models of $T_{0}$ with $M_{1}^{P} = M_{2}^{P} = N$ say, then choose any $a_{1}\in X(M_{1})$ and $a_{2}\in X(M_{2})$. Then 
the map $f$ from $M_{1}$ to $M_{2}$ which is the identity on $N$ and takes $g\cdot a_{1}$ to $g\cdot a_{2}$ for $g\in G(N)$ is visibly an isomorphism between $M_{1}$ and $M_{2}$.  (We only have to check that it preserves $F$). 
\newline
(ii) is also immediate.
\newline
(iii) is as in Lemma 2.11:   Let $A$ be a complete countable subset of (the universe of) a saturated model. We know there is a countable model , say $M_{1}$ containing $A$ and with $P(M_{1}) = P(A)$.  If $A = P(A)$ then all good types 
over $A$ are realized in $M_{1}$ (by relative categoricity). So there are only countably many.  Otherwise let $b\in X(M_{1})\cap A$.  So as $M_{1}\in dcl(P(M_{1},b)$, $A$ is contained in $dcl(P(M_{1}), b)$ so any good type $p(x)$ over $A$ 
may be assumed to be over $P(M_{1}), b)$.  If $a$ realizes $p$, then again $tp(a,b/P(M_{1}))$ is  realized in $M_{1}$, so isolated, So $p(x)$ is isolated. So $S_{*}(A)$ is countable.
\newline
(iv)   Fix an infinite cardinal $\kappa$.  Let $N_{1}$ be the following model of $T^{P}$. $G(N) = (\Z/2\Z)^{(\kappa)}$ and $I(N_{1}) = \kappa$. So by parts (i) and (ii) and Proposition 2.1. $N_{1} = M_{1}^{P}$ for some model $M_{1}$ of $T$.  Consider $Y(M_{1})$  which has cardinality $\kappa$.  Let $A = N_{1}\cup Y(M_{1})$ which is clearly complete. Let $f$ be a function which for each $\alpha< \kappa$ picks out one of the two elements of $Y_{\alpha}(M_{1})$. Note that there are $2^{\kappa}$ such $f$.
Let $p_{f}(x)$ be  the (a priori partial) type over $A$ saying that $x\in X$ and for each $\alpha<\kappa$ that the orbit of $x$ under $G_{\alpha}(M_{1})$ is precisely $f(\alpha) \in Y_{\alpha}$. 
Then by compactness $p(x)$ is consistent, so realized (in the big saturated model) by $a_{f}$ say.  We know that $``x\in X"$ isolates a complete type over $N_{1}$ (in the sense of $T$ and in the sense of $T_{0}$ too).  So $dcl(N_{1},a_{f})$ is a model $M_{1}'$ of $T$ with $P(M_{1}') = N_{1}$ and which is isomorphic to $M_{1}$ over $N_{1}$.  Moreover $Y(M_{1}') = Y(M_{1})$. It follows that $tp(a_{f}/A)$ is good.  So there are $2^{\kappa}$ good types over $A$.

\end{proof} 

\begin{Remark} (i) So by  Proposition 2.13, full $\omega$-stability (over $P$) and full stability (over $P$) are not preserved by passing to $T_{0}^{eq}$, so are not particularly robust notions.
\newline
(ii)  On the other hand, adapting \cite{Pillay-Shelah}, \cite{Shelah-classification-II}, and \cite{S-U} to the $\omega$-stable case would be useful.

\end{Remark}

\end{document}